\documentclass{article}
\usepackage{latexsym}
\usepackage{mathrsfs}
\usepackage{indentfirst}
\usepackage{amsmath,amsthm, amssymb}
\usepackage[dvips]{graphicx}
\usepackage{graphicx}
\usepackage{epsfig}
\usepackage{verbatim}

\newtheorem{theorem}{Theorem}[section]
\newtheorem{lemma}[theorem]{Lemma}
\newtheorem{corollary}[theorem]{Corollary}

\begin{document}
\title{Notes on factor-criticality, extendibility and independence number
\thanks{Corresponding email address: eltonzhang2001@yahoo.com.cn(Zan-Bo Zhang).}}
\author{Zan-Bo Zhang$^1$$^2$\thanks{Supported by the Scientific Research Foundation of Guangdong Industry Technical
College, granted No. 2007-10.}, Dingjun Lou$^2$, Xiaoyan
Zhang$^3$\thanks{Supported by Jiangsu Planned Projects for
Postdoctoral Research Funds (0602023C).}
\\
\\ $^1$Department of Computer Engineering, Guangdong
\\ Industry Technical College, Guangzhou 510300, China
\\
\\ $^2$Department of Computer Science, Sun Yat-sen
\\ University, Guangzhou 510275, China
\\
\\ $^3$School of Mathematics and Computer Science \&
\\ Institute of Mathematics,
\\ Nanjing Normal University, Nanjing 210097, China }
\date{}
\maketitle

\begin{abstract}
In this paper, we give a sufficient and necessary condition for a
$k$-extendable graph to be $2k$-factor-critical when $k=\nu/4$,
and prove some results on independence numbers in
$n$-factor-critical graphs and $k\frac{1}{2}$-extendable graphs.
$\newline$\noindent\textbf{Key words}: $k$-extendable,
$k\frac{1}{2}$-extendable, $n$-factor-critical, independence
number
\end{abstract}
\section{Introduction and preliminary results}
We consider undirected, simple, finite and connected graphs in
this paper. All terminologies and notations undefined follow that
of \cite{BM1} and \cite{LP1}.

Let $G$ be a graph, vertex set and edge set of $G$ are denoted by
$V(G)$ and $E(G)$. The number of vertices of $G$, the number of
odd components of $G$, the independence number of $G$, the edge
independence number of $G$, the connectivity of $G$ and the
minimal degree of vertices of $G$ are denoted by $\nu(G)$, $o(G)$,
$\alpha(G)$, $\alpha^\prime(G)$, $\kappa(G)$ and $\delta(G)$,
respectively. Let $G_{1}$ and $G_{2}$ be two disjoint graphs. The
\emph{union} $G_1\cup G_2$ is the graph with vertex set
$V(G_{1})\cup V(G_{2})$ and edge set $E(G_{1})\cup E(G_{2})$. The
\emph{join} $G_1\vee G_2$ is the graph obtained from $G_1 \cup
G_2$ by joining each vertex of $G_{1}$ to each vertex of $G_{2}$.
Let $X$ and $Y$ be two disjoint subsets of $V(G)$, the number of
edges of $G$ from $X$ to $Y$ is denoted by $e(X,Y)$.

A connected graph $G$ is said to be \emph{$k$-extendable} for
$0\leq k\leq (\nu-2)/2$, if it contains a matching of size $k$ and
any matching in $G$ of size $k$ is contained in a perfect matching
of $G$. The concept of $k$-extendable graphs was introduced by
Plummer in \cite{P1}. In \cite{Y1}, Yu generalized the idea of
$k$-extendibility to $k\frac{1}{2}$-extendibility for graph of odd
order. A connected graph $G$ is said to be
\emph{$k\frac{1}{2}$-extendable} if (1) for any vertex $v$ of $G$
there exists a matching of size $k$ in $G-v$, and (2) for every
vertex $v$ of $G$, every matching of size $k$ in $G-v$ is
contained in a perfect matching of $G-v$.

A graph $G$ is said to be \emph{$n$-factor-critical}, for $0\leq n
\leq \nu-2$, if $G-S$ has a perfect matching for any $S\subseteq
V(G)$ with $|S|=n$. For $n=1$, 2, that is \emph{factor-critical}
and \emph{bicritical}. The concept of $n$-factor-critical graphs
was introduced by Favaron \cite{F1} and Yu \cite{Y1},
independently.

In \cite{P1}, Plummer showed a result on connectivity in
$k$-extendable graphs.
\begin{theorem} \label{theorem:P1}
If $G$ is $k$-extendable, then $\kappa(G)\geq k+1$.
\end{theorem}

Favaron obtained a similar result for $n$-factor-critical graphs
in \cite{F1}.
\begin{theorem} \label{theorem:F1}
Every $n$-factor-critical graph is $n$-connected.
\end{theorem}

In \cite{LY1}, Lou and Yu improved the lower bound of connectivity
for $k$-extendable graph with large $k$.
\begin{theorem} \label{theorem:LY1}
If $G$ is a $k$-extendable graph on $\nu$ vertices with $k\geq
\nu/4$, then either $G$ is bipartite or $\kappa(G)\geq 2k$.
\end{theorem}

It is easy to verify that a $2k$-factor-critical graph is always
$k$-extendable, while the converse does not hold generally.
However, taking $n=2k$ in Theorem \ref{theorem:F1} and comparing
it with Theorem \ref{theorem:LY1} and Theorem \ref{theorem:P1}, we
find that the connectivity of a non-bipartitie $k$-extendable
graph increases greatly, when $k\geq \nu(G)/4$, and becomes
comparable to that of a $2k$-factor-critical graph. This fact has
motivated the authors to study the relation between non-bipartite
$k$-extendable graphs and $2k$-factor-critical graphs when $k\geq
\nu(G)/4$, and find the following theorem.
\begin{theorem} \label{theorem:ZWL1}
(Zhang et al.  \cite{ZWL1}). If $k\geq(\nu(G)+2)/4$, then a
non-bipartite graph $G$ is $k$-extendable if and only if it is
$2k$-factor-critical.
\end{theorem}

In this paper, we handle the unsettled case that $k=\nu(G)/4$.
Precisely, we give a sufficient and necessary condition for a
$k$-extendable graph with $k=\nu(G)/4$ to be $2k$-factor-critical.

In the rest of the paper we study the relationships between
independence number, factor-criticality and extendibility. Some
existing results are summarized in the following theorems.

\begin{theorem} \label{theorem:MV1}
(Maschlanka and Volkmann \cite{MV1}). Let $G$ be an $k$-extendable
non-bipartite graph. Then $\alpha(G)\leq \nu/2-k$. Moreover, the
upper bound for $\alpha(G)$ is sharp for all $k$ and $\nu$.
\end{theorem}

\begin{theorem} \label{theorem:AC1}
(Ananchuen and Caccetta \cite{AC1}). Let $G$ be a graph of even
order $\nu$ and $k$ a positive integer such that $\nu/4 \leq k
\leq \nu/2-2$, $\nu/2-k$ is even and $\delta(G)\geq \nu/2+k-1$.
Then $G$ is $k$-extendable if and only if $\alpha(G)\leq \nu/2-k$.
\end{theorem}

Some known results that will be used in our proofs are listed
below.

Let $G$ be any graph. Denote by $D(G)$ the set of vertices in $G$
which are not covered by at least one maximum matching of $G$. Let
$A(G)$ be the set of vertices in $V(G)-D(G)$ adjacent to at least
one vertex in $D(G)$, and $C(G)=V(G)-A(G)-D(G)$.

\begin{lemma} \label{lemma:GE1}
(The Gallai-Edmonds Structure Theorem \cite{LP1}). If $G$ is a
graph and $D(G)$, $A(G)$ and $C(G)$ are defined as above, then
$\newline$(a) the components of the subgraph induced by $D(G)$ are
factor-critical,
$\newline$(b) if $M$ is any maximum matching of $G$, it contains a
near-perfect matching of each component of $D(G)$, a perfect
matching of each component of $C(G)$ and matches all vertices of
$A(G)$ with vertices in distinct components of $D(G)$,
$\newline$(c) $\alpha^\prime(G)=(\nu(G)-o(D(G))+|A(G)|)/2$, where
$o(D(G))$ denotes the number of components of the subgraph induced
by $D(G)$.
\end{lemma}

\begin{lemma} \label{lemma:Y1}
(Yu \cite{Y1}). A graph $G$ of odd order is
$k\frac{1}{2}$-extendable if and only if $G \vee K_1$ is
$(k+1)$-extendable.
\end{lemma}
\begin{lemma} \label{lemma:Y2}
(Yu \cite{Y1}, Favaron \cite{F1}). A graph $G$ is
$n$-factor-critical if and only if $\nu(G)\equiv n\ (mod\ 2)$ and
for any vertex set $S\subseteq V(G)$ with $|S|\geq n$, $o(G-S)\leq
|S|-n$.
\end{lemma}
\begin{lemma} \label{lemma:Z1}
(Zhang et al. \cite{ZWL1}). If $G$ is a $k$-extendable graph, then
$G$ is also $m$-extendable for all integers $0\leq m\leq k$.
\end{lemma}
\section{Extendibility and factor-criticality}
\begin{theorem} \label{theorem:ext_crt_4k}
Let $G$ be a non-bipartite $k$-extendable graph with $\nu(G)=4k$,
then
$\newline$(1) if $\delta(G)\geq 3k$, $G$ is $2k$-factor-critical,
$\newline$(2) if $\delta(G)=2k$, $G$ is not $2k$-factor-critical,
$\newline$(3) if $2k+1\leq \delta(G) \leq 3k-1$, then $G$ is not
$2k$-factor-critical if and only if there exists a partition of
$V(G)$ into $V_1$ and $V_2$, where $|V_1|=|V_2|=2k$. Each of
$G[V_1]$ and $G[V_2]$ is composed of two factor-critical
components of size no less than 3.
\end{theorem}

\begin{proof}
Let $G$ be a non-bipartite $k$-extendable graph with $\nu(G)=4k$.
By Theorem \ref{theorem:LY1}, $\delta(G)\geq \kappa(G) \geq 2k$.
We will, one by one, discuss the three cases above.

(1) $\delta(G)\geq 3k.$ If $G$ is not $2k$-factor-critical, then
there exists a set $S\subseteq V(G)$ of size $2k$, such that $G-S$
does not have a perfect matching. Let $M_S=\{u_1v_1, u_2v_2,
\ldots, u_rv_r\}$ be a maximum matching of $G[S]$ of size $r$.
Clearly $r\leq k-1$. Hence there are at least two vertices $w_1$,
$w_2\in S$ that are not covered by $M_S$. If $w_1u_i$, $w_2v_i\in
E(G)$ for any $1\leq i \leq r$, then $(M_S\backslash \{u_iv_i\})
\cup \{w_1u_i, w_2v_i\}$ is a matching of $G[S]$ of size $r+1$,
contradicting the maximality of $M_S$. So $|\{w_1u_i, w_2v_i\}\cap
E(G)|\leq 1$. Similarly $|\{w_1v_i, w_2u_i\}\cap E(G)|\leq 1$.
Therefore $e(\{w_1,w_2\},\{u_i,v_i\})\leq 2$ for $1\leq i\leq r$.
Then we have
$$6k\leq d(w_1)+d(w_2)\leq 2r+2k+2k=4k+2r \leq 6k-2,$$ a contradiction.

(2) $\delta(G)= 2k$. Let $v$ be a vertex of degree $2k$. Then $v$
is an isolated vertex in $G-N(v)$, where $N(v)$ denotes the set of
the neighbors of $v$ in $G$. So $G-N(v)$ does not have a perfect
matching and $G$ is not $2k$-factor-critical.

(3) $2k+1\leq \delta(G)\leq 3k-1$. If there exists a partition of
$V(G)$ as stated, then $G[V_2]=G-V_1$ does not have a perfect
matching and hence $G$ is not $2k$-factor-critical.

Conversely, suppose that $G$ is not $2k$-factor-critical. Then
there exists a vertex set $S\subseteq V(G)$ of order $2k$, such
that $G-S$ does not have a perfect matching. We choose $S$ so that
$\alpha^\prime(G[S])$ has the maximum value. Clearly,
$\alpha^\prime(G[S])\leq k-1$.

Let $M_S$ be a maximum matching of $G[S]$, then there exist two
vertices $u_{1}$ and $u_{2}$ in $G[S]$ that are not covered by
$M_S$. By Lemma \ref{lemma:Z1}, $M_S$ is contained in a perfect
matching $M$ of $G$. Let $u_{i}v_{i}\in M$, where $v_{i}\in
V(G-S)$, $i=1$, 2. Let $S^{\prime}=(S\backslash \{u_{2}\})\cup
\{v_{1}\}$. Then $M_S\cup\{u_1v_1\}$ is a matching of
$G[S^\prime]$ of size $\alpha^\prime(G[S])+1$. By the choice of
$S$, $G-S^{\prime}$ has a perfect matching $M_{\bar{S^{\prime}}}$
of size $k$. Then $M_{\bar{S^{\prime}}}$ is contained in a perfect
matching $M^\prime$ of $G$. Clearly, $M^\prime\cap E(G[S^\prime])$
is a perfect matching of $G[S^\prime]$ and $M^\prime\cap E(G[S])$
is a matching of $G[S]$ of size $k-1$. Therefore,
$|M_S|=\alpha^\prime(G[S])=k-1$. Furthermore, $M\cap E(G-S)$ is a
matching of $G-S$ of size $k-1$, hence $\alpha^\prime(G-S)=k-1$.

Apply Lemma \ref{lemma:GE1} on $G[S]$ and $G-S$. Let
$C_S=C(G[S])$, $A_S=A(G[S])$, $D_S=D(G[S])$, $C_{\bar{S}}=C(G-S)$,
$A_{\bar{S}}=A(G-S)$ and $D_{\bar{S}}=D(G-S)$.

Firstly, we have $D_S$, $D_{\bar{S}}\neq \emptyset$. Let $v\in
D_S$. Then $v$ is missed by a maximum matching, say $M_S$, of
$G[S]$. Since $\delta(G)\geq 2k+1$, $v$ has at least one neighbor,
say $u$, in $V(G-S)$. By the extendibility of $G$, $M_S\cup
\{uv\}$ is contained in a perfect matching $M_0$ of $G$. Moreover
$M_0\cap E(G-S)$ is a maximum matching of $G-S$, which misses $u$.
So $u\in D_{\bar{S}}$. Therefore, $e(D_S,A_{\bar{S}}\cup
C_{\bar{S}}) =e(D_{\bar{S}},A_{S}\cup C_{S}) = 0$.

Now we prove that $A_S\cup C_S=A_{\bar{S}} \cup
C_{\bar{S}}=\emptyset$. By contradiction, suppose that at least
one of the equalities does not hold, say $A_S\cup C_S\neq
\emptyset$. If $A_{\bar{S}}\cup C_{\bar{S}}= \emptyset$, then
$D_S$ is a cut set of $G$ of size less than $2k$, contradicting
$\kappa(G)\geq 2k$. Hence we can assume that $A_{\bar{S}}\cup
C_{\bar{S}}\neq \emptyset$. Then both $D_S\cup A_{\bar{S}}$ and
$D_{\bar{S}}\cup A_S$ are cut sets of $G$. Thus we have $|D_S\cup
A_{\bar{S}}|\geq 2k$ and $|D_{\bar{S}}\cup A_S|\geq 2k$. However
$|D_S\cup A_{\bar{S}}|+|D_{\bar{S}}\cup A_S|=
\nu(G)-|C_S|-|C_{\bar{S}}|\leq 4k$. So all equalities must hold,
that is, $|D_S|+|A_{\bar{S}}|=|D_{\bar{S}}|+|A_S|=2k$ and
$C_S=C_{\bar{S}}=\emptyset$.

By our assumption, we must have $A_S, A_{\bar{S}}\neq \emptyset$.
Suppose that there is an edge $e\in E(G)$ connecting two vertices
in $A_S$. Take a maximum matching $M_{\bar{S}}$ of $G-S$, by the
extendibility of $G$, $M_{\bar{S}}\cup \{e\}$ is contained in a
perfect matching $M_1$ of $G$. But then $M_1\cap E(G[S])$ is a
maximum matching of $G[S]$ containing $e$, contradicting Lemma
\ref{lemma:GE1} (b). Hence $A_S$ spans no edge of $G$. Then for
any $w\in A_S$, $d(w)\leq |A_{\bar{S}}|+|D_S|=2k$, contradicting
$\delta(G)\geq 2k+1$. Therefore $A_S=\emptyset$ and similarly
$A_{\bar{S}}=\emptyset$.

Now we have $D_S=S$ and $D_{\bar{S}}=V(G)\backslash S$. By Lemma
\ref{lemma:GE1} (a), each component of $G[S]$ and $G-S$ is
factor-critical. By Lemma \ref{lemma:GE1} (c),
$o(D_S)=o(D_{\bar{S}})=2$, hence each of $G[S]$ and $G-S$ consists
of two factor-critical components. Finally, since $\delta(G) \geq
2k+1$, all the components must have size at least 3.
\end{proof}

\section{Factor-criticality and independence number}
The lower bound in the theorem below has been proved in a remark
in \cite{F2}. Note that since every $2k$-factor-critical graph is
$k$-extendable, it is a straight consequence of Theorem
\ref{theorem:MV1} when $n$ is even. The sharpness can be verified
by the graph $G=K_{(\nu+n)/2} \vee ((\nu-n)/2)K_1$.
\begin{theorem} \label{theorem:crt_ind_1}
Let $G$ be a $n$-factor-critical graph of order $\nu$. Then
$\alpha(G)\leq (\nu-n)/2$. The bound for $\alpha(G)$ is sharp.
\end{theorem}

Again in a remark in \cite{F2}, Favaron pointed out that the
conditions in Theorem \ref{theorem:AC1} yield
$2k$-factor-criticality rather than $k$-extendibility. Here we
prove a general version for all $n$-factor-critical graphs.

\begin{theorem} \label{theorem:ind_crt_1}
Let $G$ be a graph on $\nu$ vertices, $n$ a positive integer
satisfying $\nu \equiv n\ (mod\ 2)$, $\delta(G)\geq (\nu+n)/2-1$
and $\alpha(G)\leq (\nu-n)/2$. Then $G$ is not $n$-factor-critical
if and only if $(\nu-n)/2$ is odd, $G=G_0 \vee (G_1 \cup G_2)$,
where $\nu(G_0)=n$ and $G_1=G_2=K_{(\nu-n)/2}$. The bounds for
$\delta(G)$ and $\alpha(G)$ are sharp.
\end{theorem}
\begin{proof}
Suppose that $G$ is not $n$-factor-critical. By Lemma
\ref{lemma:Y2}, there exists $S\subseteq V(G)$ of size at least
$n$, such that $o(G-S)>|S|-n$. By parity we have $o(G-S)\geq
|S|-n+2$.

Let $G_1$ be an odd component of $G-S$ of the minimum size, and
$v\in V(G_1)$. Then $\nu(G_1)\leq (\nu-|S|)/(|S|-n+2)$. Hence
\begin{equation} \label{equation:1}
\frac{\nu+n}{2}-1\leq \delta(G)\leq d(v)\leq \nu(G_1)+|S|-1 \leq
\frac{\nu-|S|}{|S|-n+2}+|S|-1.
\end{equation}

Solving $|S|$ from (\ref{equation:1}) we have $|S|\leq n$ or
$|S|\geq (\nu+n)/2-1$. If $|S|\geq (\nu+n)/2-1$, then $o(G-S)\geq
(\nu+n)/2-1-n+2=(\nu-n)/2+1$. Selecting one vertex from each odd
component of $G-S$ we form an independent set of $G$ of order no
less than $(\nu-n)/2+1$, contradicting $\alpha(G)\leq (\nu-n)/2$.
Therefore we can assume that $|S|\leq n$. However $|S|\geq n$ by
our selection of $|S|$, so we have $|S|=n$ and all equalities in
(\ref{equation:1}) must hold. Hence $G-S$ has exactly two
components $G_1$ and $G_2$ of odd size $(\nu-n)/2$, $G_1$ and
$G_2$ must be $K_{(\nu-n)/2}$ and all vertices in $V(G_1)\cup
V(G_2)$ are adjacent to all vertices in $S$. Thus $G=G_0 \vee (G_1
\cup G_2)$, where $G_0=G[S]$ is of order $n$ and
$G_1=G_2=K_{(\nu-n)/2}$.

On the contrary if $G$ satisfies the conditions stated, then
$G-V(G_0)$ does not have a perfect matching, so $G$ is not
$n$-factor-critical.

The graph $G=((\nu-n)/2+1)K_1 \vee K_{(\nu+n)/2-1}$ shows that the
bound for $\alpha(G)$ is sharp, while the sharpness of the bound
for $\delta(G)$ can be verified by the graph $G=(K_3 \cup
((\nu-n)/2)-1)K_1) \vee K_{(\nu+n)/2-2}$.
\end{proof}
Combining Theorem \ref{theorem:crt_ind_1} and Theorem
\ref{theorem:ind_crt_1} we have the following result.
\begin{theorem} \label{theorem:ind_eq_crt}
Let $G$ be a graph on $\nu$ vertices, $n$ a positive integer
satisfying $\nu \equiv n\ (mod\ 2)$, $\delta(G)\geq (\nu+n)/2-1$.
Suppose that $G$ can not be expressed as $G=G_0 \vee (G_1 \cup
G_2)$, where $(\nu-n)/2$ is odd, $\nu(G_0)=n$ and
$G_1=G_2=K_{(\nu-n)/2}$. Then $G$ is $n$-factor-critical if and
only if $\alpha(G)\leq (\nu-n)/2$.
\end{theorem}

By Theorem \ref{theorem:ZWL1} and Theorem
\ref{theorem:ext_crt_4k}, when $k\geq \nu(G)/4$ and $\delta(G)\geq
\nu(G)/2+k-1$, $G$ is $k$-extendable if and only if $G$ is
$2k$-factor-critical, only except that when $k=\nu(G)/4$ is odd
and $G=(K_k \cup K_k)\vee(K_k \cup K_k)$, $G$ is $k$-extendable
but not $2k$-factor-critical. But then $G$ is exactly the
exceptional graph in Theorem \ref{theorem:ind_eq_crt}. Hence we
have the following corollary which has Theorem \ref{theorem:AC1}
as part of it.

\begin{corollary}
Let $G$ be a graph with even order $\nu$, $k$ a positive integer
such that $\nu/4 \leq k \leq \nu/2 -1$ and $\delta(G)\geq
\nu/2+k-1$. Suppose that $G$ can not be expressed as $G=G_0 \vee
(G_1 \cup G_2)$, where $\nu/2-k$ is odd, $\nu(G_0)=2k$ and
$G_1=G_2=K_{\nu/2-k}$. Then the following are equivalent.
$\newline$ (1) $G$ is $k$-extendable,
$\newline$ (2) $G$ is $2k$-factor-critical,
$\newline$ (3) $\alpha(G)\leq \nu/2-k$.
\end{corollary}
\section{Extendibility and independence number}

In this section we generalize Theorem \ref{theorem:MV1} for
$k\frac{1}{2}$-extendable graphs.

\begin{theorem} \label{theorem:k1/2_ind}
Let $G$ be an $k\frac{1}{2}$-extendable graph on $\nu$ vertices.
Then $\alpha(G)\leq (\nu-1)/2-k$. Moreover, the upper bound for
$\alpha(G)$ is sharp for all $k$ and $\nu$.
\end{theorem}
\begin{proof}
By Lemma \ref{lemma:Y1}, $G\vee K_1$ is $(k+1)$-extendable. By
Theorem \ref{theorem:MV1}, $\alpha(G\vee K_1) \leq (\nu+1)/2 -
(k+1)=(\nu-1)/2-k$. Furthermore, any independent set $S$ of $G\vee
K_1$ with $|S|>1$ can not contain the vertex in $K_1$. Hence $S$
is also an independent set of $G$. So $\alpha(G)=\alpha(G\vee K_1)
\leq(\nu-1)/2-k$.

To see that the bound is sharp, consider the graph
$G=((\nu-1)/2-k)K_1 \vee K_{(\nu+1)/2+k} $.
\end{proof}


\end{document}